\documentclass[12pt]{amsart}

\usepackage{amsfonts,amssymb}
\usepackage{hyperref, graphicx}
\usepackage{epsfig}
\usepackage{latexsym}
\usepackage{amsmath,amsthm}
\usepackage{mathrsfs}
\usepackage[all,cmtip]{xy}
\theoremstyle{plain}
\newtheorem{theorem}{Theorem}[section]
\newtheorem{lemma}[theorem]{Lemma}
\newtheorem{definition}[theorem]{Definition}
\newtheorem{proposition}[theorem]{Proposition}
\newtheorem{cor}[theorem]{Corollary}
\newtheorem{remark}[theorem]{Remark}

\numberwithin{equation}{section}

\newcommand{\ra}{\longrightarrow}
\newcommand{\C}{\mathbb{C}}
\newcommand{\R}{\mathbb{R}}

\newcommand{\sheaf}[1]{\ensuremath{\mathcal{#1}}}
\newcommand{\rst}[1]{\ensuremath{\sigma^{#1}}}
\newcommand{\catRP}{\mathbf{RPar}^\mathrm{ss}_{\tau_p}(X, \sigma_X)}
\newcommand{\catQP}{\mathbf{QPar}^\mathrm{ss}_{\tau_p}(X, \sigma_X)}
\newcommand{\catP}{\mathbf{Par}^\mathrm{ss}_{\tau_p}(X)}
\newcommand{\parE}{\overline{\partial}_E}

\hypersetup{colorlinks, citecolor=blue, filecolor=black, linkcolor=red}

\begin{document}
\baselineskip=15.5pt

\title{A gauge theoretic aspect of parabolic bundles over real curves}
\author{Sanjay~Amrutiya}
\address{Department of Mathematics, IIT Gandhinagar,
 Near Village Palaj, Gandhinagar - 382355, India}
 \email{samrutiya@iitgn.ac.in}
\author{Ayush~Jaiswal}
\address{Department of Mathematics, IIT Gandhinagar,
 Near Village Palaj, Gandhinagar - 382355, India}
\email{ayush.jaiswal@iitgn.ac.in}
\thanks{SA was supported by the SERB-DST under project no. YSS/2015/001182 and MTR/2018/000475.}
\date{}

%%%%%%%%%%%%%%%%%%%%%%%%%%%%%   Main Document   %%%%%%%%%%%%%%%%%%%%%%%%%%%%%%%%%%
\begin{abstract}
\noindent In this article, we study the gauge theoretic aspects of real and quaternionic parabolic
bundles over a real curve $(X, \sigma_X)$, where $X$ is a compact Riemann surface and $\sigma_X$ is
an anti-holomorphic involution. For a fixed real or quaternionic structure on a smooth parabolic 
bundle, we examine the orbits space of real or quaternionic connection under the appropriate gauge group.
The corresponding gauge-theoretic quotients sit inside the real points of 
the moduli of holomorphic parabolic bundles having a fixed parabolic type on a compact Riemann surface $X$ .
\end{abstract}
%--------------------------------------------------------------------------------%
\maketitle

\section{Introduction}\label{intro}
In \cite{NS65}, Narasimhan and Seshadri proved that the vector bundles
associated with irreducible unitary representations of the fundamental group of a compact Riemann
surface are precisely the stable vector bundles on a compact Riemann surface. In \cite{Do83}, 
Donaldson proved the Narasimhan-Seshadri theorem using the results of \cite{U82}. When a compact 
Riemann surface $X$ is equipped with an anti-holomorphic involution $\sigma_X$, an analogue of Narasimhan-Seshadri
theorem for real and quaternionic bundles is studied in \cite{Sch17}, see also \cite{BHH10}. 

The notion of parabolic bundles on compact Riemann surfaces was first introduced by
C. S. Seshadri and their moduli was constructed in \cite{MS80}, using GIT, by 
Mehta and Seshadri. In \cite{MS80}, they have proved that stable parabolic bundles 
of degree zero on a compact Riemann surface are precisely those vector bundles associated with irreducible unitary representations of the fundamental group of 
a punctured Riemann surface. In \cite{Bi91}, Biquard improved (allowing real parabolic weights)
the result of Mehta and Seshadri following \cite{Do83} by considering appropriate Sobolev spaces 
using results of \cite{LM85}. See also \cite{Ko93, Po93, DW97} for a gauge-theoretic approach to 
parabolic bundles. The parabolic bundles over a real curve $(X, \sigma_X)$ is studied in
\cite{Am14a, Am14b, BS20}. In \cite{BS20}, Biswas and Schaffhauser established a bijective 
correspondence between the isomorphism classes of polystable real and quaternionic parabolic vector 
bundles and the equivalence classes of real and quaternionic unitary representations of the orbifold
fundamental group of $(X, \sigma_X)$.

This paper studies the gauge-theoretic aspects of parabolic bundles over a real curve. 
Section \ref{sec-2} reviews some basic concepts and results concerning parabolic bundles on 
compact Riemann surfaces. In Section \ref{sec-3}, we examine the stability of real (resp. quaternionic)
parabolic bundles and $S$-equivalence classes of such bundles. In Section \ref{sec-4}, we study the 
induced real structure on the space of connection and parabolic gauge group. We show that the corresponding
quotients parametrize the real $S$-equivalence classes of semistable real (resp. quaternionic) parabolic
bundles.

\section{Preliminaries}\label{sec-2}
This section recalls some basic notions and results about parabolic bundles. More 
details can be found in \cite{MS80, Bi91}.

\subsection{Parabolic bundles}
Let $X$ be a compact Riemann surface and $S$ a finite subset of $X$.
Let $E$ be a smooth complex vector bundle of rank $r$ on $X$.
A quasi-parabolic structure on $E$ at a point $x\in S$
is a strictly decreasing flag
$$
E_x \; = \; F^1E_x \; \supseteq \; F^2E_x \; \dots \; F^{k_x}E_x \;
\supseteq \; F^{k_x+1}E_x \; = \; 0
$$
of linear subspaces in $E_x$. We define
$$
r_j^x = \dim F^jE_x - \dim F^{j+1}E_x.
$$
The integer $k_x$ is called the flag's length, and the
sequence $(r_1^x,\dots r_{k_x}^x)$ is called the flag type.
The points in $S$ are called parabolic points.

A parabolic srtucture in $E$ at $x$ is a quasi-parabolic
structure at $x$ as above, together with a sequence of
real numbers $0\leq \alpha^x_1 < \dots < \alpha^x_{k_x} < 1$.
We call $r_1^x, \dots, r_{k_x}^x$ the multiplicities of
$\alpha^x_1,\dots, \alpha^x_{k_x}$. The $\alpha_j$ are called the
weights, and we set
$$
d_x(E) = \displaystyle \sum_{j+1}^{k_x} r_j\alpha_j ~~\mbox{and}~~
\mathrm{wt}(E) = \displaystyle \sum_{x\in S} d_xE.
$$

We say that $E$ is a holomorphic parabolic bundle with a parabolic structure on $S$ if  
given a parabolic structure on the underlying smooth complex vector bundle $E$ at each point $x\in S$. 
We denote it by $E_\bullet = (E, {F^iE(x)}, \alpha^x_i)_{x\in P}$. 

By a \emph{parabolic type} $\tau_p$, we mean a fixed
flag type $(r_1,\dots r_{k_x})$, fixed weights $0\leq \alpha^x_1 < \dots < \alpha^x_{k_x} < 1$
and degree $d$.

The parabolic degree is defined by
\begin{equation}\label{p-degree:eq:1}
\mathrm{par}\mbox{deg}(E) = \mbox{deg}(E) + \mathrm{wt}(E).
\end{equation}
We set
\begin{equation}\label{p-slope:eq:1}
\mathrm{par}\mu(E) = \frac{\mathrm{par}\mbox{deg}(E)}{\mbox{rank}(E)}.
\end{equation}

A holomorphic parabolic bundle $E_\bullet$ is called \emph{semi-stable} (resp. \emph{stable}) if for 
all sub-bundles $F$ of $E$, we have $\mathrm{par}\mu(F) \leq \mathrm{par}\mu(E)$ 
(resp. $\mathrm{par}\mu(F) \leq \mathrm{par}\mu(E)$), 
where $F$ has induced parabolic structure from $E_\bullet$.

Let $M_X^\mathrm{ss}(\tau_p)$ be the set of $S$-equivalence classes of semi-stable parabolic bundles 
on $X$ having parabolic type $\tau_p$.

\begin{theorem}\cite{MS80}
There exists a natural structure of a normal projective variety on $M_X^\mathrm{ss}(\tau_p)$ of
dimension $r^2(g-1) + 1 + \sum_{x\in S} \frac{1}{2}(r^2-\sum_{i = 1}^{k_x} (r_i^x)^2)$.
\end{theorem}

\begin{theorem}\cite{MS80}
A holomorphic parabolic bundle $E_\bullet$ of parabolic degree $0$ is stable if and only if there 
is an irreducible unitary representation $\rho\colon  \pi_1(X\setminus S)\ra U(r)$ such that 
$E_\bullet \cong E^\rho_\bullet$, where $E^\rho_\bullet$ is a holomorphic parabolic bundle 
associated to $\rho$.
\end{theorem}

\subsection{Gauge theoretic formulation}
For a smooth complex parabolic vector bundle $E$ of rank $r$ on $X$ with parabolic type $\tau_p$, 
let $\mathcal{C}$ denote the space of holomorphic structure on $E$, more precisely, the space of 
operators
$$
\parE \colon A^0(E)\ra A^{0,1}(E); \quad 
\bar{\partial}_E(fs) = f\bar{\partial}_E(s) + (\bar{\partial}f)s \,.
$$
Throughout the article, $A^p(E)$($A^{p,q}(E)$, respectively) denotes space of smooth $p$(smooth $(p,q)$, respectively) forms on base space with values in bundle.
Then, there is a bijective correspondence between $\mathcal{G}_{\mathrm{par}}$-orbits in $\mathcal{C}_E$ and 
the isomorphism classes of holomorphic parabolic bundles on $X$ having parabolic type $\tau_p$, where
$$
\mathcal{G}_{\mathrm{par}} := \{g\in C^\infty(\mathrm{Aut}(E))\;|\; g~\mbox{respect the flag of}~ E_x, ~\mbox{for all}~ x\in S\}\,.
$$
Let us first review Biquard's formulation of Mehta-Seshadri theorem \cite{MS80, Bi91}. 

\subsection*{Sobolev spaces}
Let $E_\bullet$ be a smooth complex vector bundle on $X$ with parabolic structure over $S$. The weighted 
Sobolev norm for $s\in A^\ell(E)$ is defined as
$$
\| s \|_{W^{k, p}_\delta}:= \sum_{j=0}^k \| \bigtriangledown^j s \|_{L^p_\delta}\,,
$$
where $\| \cdot \|_{L^p_\delta}$ weighted $L^p$ norm with weight $\delta$ (for more details, see \cite{LM85, Bi91} ). 
We denote by $W^{k, p}_\delta(E)$  the completion of $A^\ell(E)$ with respect to the weighted Sobolev 
norm $\| \cdot \|_{W^{k, p}_\delta}$.

For $\delta = k - 2/p$, we let $W^p_k := W^{k, p}_\delta$. 

Let 
$F = \mathbb{D} \times \C^r$ be a trivial vector bundle with the standard metric. 
Let $0\leq \alpha_1 < \alpha_2 \cdots < \alpha_\ell < 1$ be the fixed real numbers, let $\alpha$ be 
the matrix
\[
\alpha = \begin{pmatrix}
\tilde{\alpha}_1 & & \\
& \ddots & \\
& & \tilde{\alpha}_r
\end{pmatrix}
\]
where $0\leq \tilde{\alpha}_1 \leq \tilde{\alpha}_2 \cdots \leq \tilde{\alpha}_r < 1$ in which $\alpha_i$ are re-labeled
according to their multiplicities.

Consider the decomposition of $F$ by the eigenspaces $E_{\alpha_i}$ of 
$\alpha = \mathrm{diag}(\tilde{\alpha}_1, \dots, \tilde{\alpha}_r)$. Then, for any $u\in \mathrm{End}(E)$, 
we have
$$
u^D\in \bigoplus_i \mathrm{Hom}(E_{\alpha_i}, E_{\alpha_i}) \quad \mbox{and} \quad 
u^H\in \bigoplus_{i\neq j} \mathrm{Hom}(E_{\alpha_i}, E_{\alpha_j})\,.
$$
Consider the space 
$$
\mathcal{D}^p_k(\mathrm{End}(F)) = 
\{u\in \mathrm{End}(F) \;|\; u^D\in L^p_k(\mathrm{End}(F)), \; u^H\in W^p_k(\mathrm{End}(F)) \}
$$
with the norm $\|u\|_{\mathcal{D}^p_k} = \|u^D\|_{L^p_k} + \|u^H\|_{W^p_k}$.

Let $E_\bullet$ be a smooth complex parabolic vector bundle on $X$ with parabolic structure over $S$. 
Let $x\in S$, and let $z$ be a local coordinate on $X$ at $p$ such that $z(x)=0$. Let $\{s_i\}$ be a local frame 
of $E$ at $x$. We say that a local frame $\{s_i\}$ of $E$ at $p$ respect the flag structure at $x$ if $F^jE(x)$ 
is generated by $\{s_i(x)\}_{i\geq r-\dim F^jE(x)+1}$.

Consider a Hermitian metric $h$ in $E_{|X-S}$. We say that a metric $h$ is $\alpha$-adapted if for any 
parabolic point $x\in S$, the following holds: Choose a local coordinate $z$ and a local frame $\{e_i\}$ of 
$E$ near $x$ which respect the flag structure at $x$. Then, there is a gauge transformation $g$ near $x$
such that in the local frame $\{g(e_i)\}$, one has
\begin{equation}\label{h-matrix}
h = \begin{pmatrix}
|z|^{2\alpha_1} & & 0\\
& \ddots & \\
0 & & |z|^{2\alpha_\ell}
\end{pmatrix}\,.
\end{equation}

Let $(E_\bullet, h)$ be a smooth Hermitian vector bundle on $X$ with parabolic structure over $S$, where
$h$ is an adapted Hermitian metric with respect to the given parabolic structure over $S$.

Recall that the Chern connection associated with the adapted hermitian metric on holomorphic parabolic bundle $E$ is
$$
d^h=d+h^{-1}\partial h=d+\alpha\frac{dz}{z}
$$

With respect to the adapted frame $(\frac{e_i}{|z|^{\alpha^i}})$, we have
\[
\begin{array}{ll}
d^h(\frac{e_i}{|z|^{\alpha_i}}) 
& = e_i d(\frac{1}{|z|^{\alpha_i}})+\frac{1}{|z|^{\alpha_i}}d^h(e_i) \\
& \\
& = -e_i\cdot\frac{\alpha_i}{2|z|^{\alpha_i+2}}(\overline{z}dz+zd\overline{z})+e_i\frac{\alpha_i}{z|z|^{\alpha_i}}dz \\
& \\
& = \frac{e_i}{|z|^{\alpha_i}}\frac{\alpha_i}{2}(\frac{dz}{z}-\frac{d\overline{z}}{\overline{z}})\,.
\end{array}
\]
From this, it follows that $d^h = d+\frac{\alpha}{2}(\frac{dz}{z}-\frac{d\overline{z}}{\overline{z}})$ 
in the adapted frame $(\frac{e_i}{|z|^{\alpha^i}})$.

Let $\mathcal{A}$ denote the space of $h$-unitary connections associated with the holomorphic structure 
$\parE\in \mathcal{C}$.
Consider the unitary gauge group of $(E, h)$ defined by
$$
\mathcal{G}_h = \{g \in \mathcal{G}_{\mathrm{par}} \;\mid\; g_{|_{X\text{--} S}} ~\mbox{is}~ h\text{-}\mbox{unitary}\}.
$$
Let $\mathcal{C}^p$ be the space of Daulbault operator $\parE$ of class $L^p_1$ on $X-S$ and is of the form
$$
\bar{\partial} - \frac{1}{2}\alpha\frac{d\bar{z}}{\bar{z}} + a
$$
near $x\in S$ in any local frame adapted to $E$ with $a\in \mathcal{D}^p_1$.
Let $\mathcal{G}_\C^p$ be the space of Sobolev gauge transformations of $E$ of class $L^p_2$ on $X-S$
and of class $\mathcal{D}^p_2$ near $x\in S$.

Let $\mathcal{A}^p$ be the space of $h$-unitary Sobolev connections on $E$ of class $L^p_1$ on $X-S$ 
and is of the form
$$
d + \alpha\frac{dz}{z} + a
$$
near $x\in S$ in any local frame adapted to $E$ with $a\in \mathcal{D}^p_1$. We denote by 
$\mathcal{G}^p_h$ a group of unitary Sobolev gauge transformations. 

The action of the Lie group $\mathcal{G}^p$ on a connection $A = D + a \in \mathcal{A}^p$ is given by
$$
g(D + a) = D + gag^{-1} - (Dg)g^{-1}\,.
$$

The curvature of a connection $A = D + a \in \mathcal{A}^p$ is given by
$$
F_A = F_{D} + Da + \frac{1}{2}[a, a]\,.
$$
If $p\in (1, 2)$, then $D^p_1(A^1(\mathfrak{u}(E, h)))=L^p_1(A^1(\mathfrak{u}(E, h)))$, and 
hence we have the curvature map $F\colon \mathcal{A}^p \ra L^p(A^2(\mathfrak{u}(E, h)))$ 
(\cite[Lemma 1.1]{U82}).

Let $p\in (1, 2)$ satisfying
\begin{equation}\label{eq-1}
p < 
\left \{ \begin{array}{cc}
\frac{2}{2 + \alpha_j - \alpha_i} & \mbox{if} \quad \alpha_i > \alpha_j; \\
& \\
\frac{2}{1 + \alpha_j - \alpha_i} & \mbox{if} \quad \alpha_i < \alpha_j 
\end{array} \right.
\end{equation}
Then, the operator 
$$
\bar{\partial}_E \colon D^p_2(\mathrm{End}(E))\ra D^p_1(A^{0, 1}\otimes \mathrm{End}(E))
$$
is Fredholm. Using Fredholmness of this operator, it follows that any operator 
$\bar{\partial}_E \in \mathcal{C}^p$ 
is equivalent under the complex gauge group $\mathcal{G}^p_\C$ to a smooth operator on $X$ 
(i.e. which is in $\mathcal{C}$) \cite[Proposition 2.8]{Bi91} (cf. \cite[Lemma 14.8]{AB82}).
There are bijective correspondences
$$
\mathcal{A}^p/\mathcal{G}_h^p \simeq \mathcal{A}/\mathcal{G}_h \simeq \mathcal{C}/\mathcal{G}_{\mathrm{par}}.
$$

\begin{theorem}\cite{Bi91}\label{Biquard-thm}
Let $\sheaf{E}_\bullet$ be an indecomposable parabolic bundle with an adapted Hermitian metric $h$.
Then $\sheaf{E}_\bullet$ is parabolic stable if and only if there exists on $\sheaf{E}$ a connection 
$A\in \mathcal{A}$ satisfying
$$
\star F_A = -2\pi \sqrt{-1} \mathrm{par}\mu(E)\,.
$$
This connection is unique up to the action of the gauge group $\mathcal{G}_h$.
\end{theorem}

Let 
$$\mathcal{C}_{\mathrm{s}}:= \{\bar{\partial}_E \in \mathcal{C}\;|\; (E_\bullet, \bar{\partial}_E) ~ 
\mbox{is stable parabolic bundle}\}$$ and 
$$
\mathcal{A}^p_\mathrm{ss}:= F^{-1}(2\pi \sqrt{-1}\mathrm{par}\text{-}\mu(E))
$$
$$
\mathcal{A}^p_s:= \{A\in \mathcal{A}^p_\mathrm{ss} \;|\; d_A ~ \mbox{is irreducible}\}
$$ 
Then, we have the following commutative diagram
\[
\xymatrix{
\mathcal{C}_{\mathrm{s}}\ar[r]^\iota \ar[d]_q & \mathcal{A}^p_s \ar[d]^\pi \\
M_X^s(\tau_p) \ar[r]_\varphi & \mathcal{A}^p_s/\mathcal{G}^p_h \;.
}
\]
The set $M_X^s(\tau_p) := \mathcal{C}^{\mathrm{s}}/\mathcal{G}_{\mathrm{par}} \cong \mathcal{A}^p_s/\mathcal{G}^p_h$
has a natural structure of K\"ahler manifold.
In fact, Konno \cite{Ko93} studied the moduli of more general objects, namely parabolic Higgs bundles, using the
weighted Sobolev spaces defined by Biquard \cite{Bi91}. 

\subsection{Vector bundles on real curves}\label{real}
By a real curve we will mean a pair $(X, \sigma_X)$, where
$X$ is a Riemann surface, and $\sigma_X$ is an anti-holomorphic
involution on $X$. Let $\sigma_{\C} : \C\ra \C$ be the conjugate
map $z\mapsto \bar{z}$. 

A \emph{real (resp. quaternionic) holomorphic vector bundle}
\index{vector bundle! real}
$E\ra X$ is a holomorphic vector bundle, together with
an anti-holomorphic involution (resp. anti-involution) $\rst{E}$ of the total
space $E$ making the diagram
\[
\xymatrix{
E \ar[r]^{\rst{E}} \ar[d] & E \ar[d] \\
X \ar[r]_{\sigma_X} & X
}
\]
commutative, and such that, for all $x\in X$, the map
$\rst{E}|_{E(x)} \,:\, E(x)\ra E(\sigma_X(x))$ is
$\C$-antilinear:
$$
\rst{E}(\lambda\cdot \eta) =
\Bar{\lambda}\cdot \rst{E}(\eta), ~
\mbox{for all}~ \lambda\in \C~ \mbox{and all}~
\eta \in E(x).
$$
We refer to the map $\rst{E}$ as the \emph{real structure} of $E$.
Giving a real structure $\rst{E}$ on $E$ is equivalent to giving a 
$\C$-linear isomorphism $\phi\colon \sigma_X^*\overline{E}\ra E$
such that $\sigma_X^*\overline{\phi}\circ \phi = \mathrm{Id}_E$.

A homomorphism between two real bundles
$(E, \rst{E})$ and $(E', \rst{E'})$
is a homomorphism
$$
f \,:\, E\ra E'
$$
of holomorphic vector bundles over $X$ such that
$f\circ \rst{E} = \rst{E'}\circ f$.

A holomorphic subbundle $F$ of a real holomorphic
vector bundle $E$ is said to be a real subbundle of
$E$ if $\rst{E}(F) = F$.

We refer to \cite{BHH10} for topological classification 
of real and quaternionic bundles. See also \cite{Sch12} for discussion on the stability of such bundles
over a real curve.

\section{Parabolic bundles over real curves}\label{sec-3}
Let $(E, \rst{E})$ be a smooth real (resp. quaternionic) vector bundle over a real curve 
$(X, \sigma_X)$. Let $S$ be a finite subset of $X$ such that $\sigma_X(S) = S$.

\begin{definition}\label{def-par-real-quat}
A \emph{parabolic structure} on $(E, \rst{E})$ over $S$ is a quasi-parabolic structure on
$(E, \rst{E})$ over $S$:
\begin{itemize}
\item for each $x\in S$, there is a strictly decreasing flag
\begin{equation*}
E(x) = F^1E(x)\supset F^2E(x)\supset \dots \supset F^{k_x}E(x) \supset F^{k_x+1}E(x) = 0
\end{equation*}
of linear subspaces in $E(x)$ satisfying $\rst{E}_x(F^iE(x)) = F^iE(\sigma_X(x))$
\end{itemize}
together with a sequence of real numbers $0\leq \alpha^x_1 < \dots < \alpha^x_{k_x} < 1$, 
with the following property:
\begin{itemize}
 \item the weights over $x$ and
$\sigma_X(x)$ are same.
\end{itemize}
\end{definition}

A smooth real (resp. quaternionic) vector bundle $(E, \rst{E})$ together with a parabolic structure
as in Definition \ref{def-par-real-quat} will be referred to as a smooth real (resp. quaternionic) parabolic
vector bundle, and we denoted it by $(E_\bullet, \rst{E})$. 
A real (resp. quaternionic) holomorphic vector bundle $(E, \rst{E})$ together with a parabolic structure
as in Definition \ref{def-par-real-quat} will be referred to as a real (resp. quaternionic) parabolic
vector bundle. 

\begin{remark}\rm{
Let $E_\bullet$ be a holomorphic parabolic bundle on $(X, S)$. Then, $\sigma_X^*\overline{E}$ gets an induced 
parabolic structure, and the resulting holomorphic parabolic bundle is denoted by 
$\sigma_X^*\overline{E}_\bullet$. If $(E_\bullet, \rst{E})$ is a real parabolic bundle over $(X, \sigma_X)$, then
there is an isomorphism $\phi\colon E_\bullet \ra \sigma_X^*\overline{E}{_\bullet}$ of holomorphic
parabolic bundles such that $\sigma^*\overline{\phi}\circ \phi = \mathrm{Id}_E$ (see \cite{BS20}). 
A similar statement holds in the quaternionic case.
}
\end{remark}

Let us recall the definition of real parabolic 
semi-stable bundles over a real curve (see \cite{Am14a, BS20}). A real parabolic bundle
$(E, \rst{E})$ is called \emph{real semistable} if
for every real parabolic subbundle $F$ of $E$, we have
\begin{equation}\label{parsc}
\mathrm{p}\mu(F)\leq \mathrm{p}\mu(E).
\end{equation}
We say that a real parabolic bundle $(E, \rst{E})$ is \emph{real stable} if
the inequality \eqref{parsc} is strict, i. e., $\mathrm{p}\mu(F) < \mathrm{p}\mu(E)$
for every proper real subbundle $F$ of $E$.

\begin{proposition}\label{prop-rsp-hsp}
Let $(E, \rst{E})$ be a real (resp. quaternionic) semistable parabolic bundle on $(X, \sigma_X)$. 
Then, the underlying holomorphic parabolic bundle $E_\bullet$ is parabolic semistable.
\end{proposition}
\begin{proof}
Let $\phi\colon E_\bullet \ra \sigma^*\overline{E}_\bullet$ be an isomorphism of holomorphic
parabolic bundles such that $\sigma^*\overline{\phi}\circ \phi = \mathrm{Id}_E$.
If $E_\bullet$ is not parabolic semistable, then by \cite[Theorem 8]{Se82}, there exists a unique maximal
destabilizing subbundle $F$ of $E$. Note that $\phi(\sigma_X^*\overline{F})$ and $F$ 
are subbundles of $E$ having same rank and parabolic degree with respect to the induced 
parabolic structure. Since $F$ is the maximal destabilizing subsheaf of $E$ (for parabolic 
semistability), it follows that $\phi(\sigma_X^*\overline{F})$ is the maximal destabilizing subsheaf of 
$\phi(\sigma_X^*\overline{E})$. Hence, by the uniqueness, we have $\phi(\sigma_X^*\overline{F}) = F$. 
Since $(E, \rst{E})$ is real (resp. quaternionic) semistable parabolic bundles, we have 
$\mathrm{p}\mu(F) \leq \mathrm{p}\mu(E)$, which is a contradiction.
\end{proof}

The following result is a generalization of \cite[Proposition 2.7]{Sch12} to real parabolic bundles. 
The proof is identical, with some additional arguments.

\begin{proposition}\label{prop-rsp}
Let $(E, \rst{E})$ be a real (resp. quaternionic) stable parabolic bundle on $(X, \sigma_X)$. Then, 
one of the following holds:
\begin{enumerate}
\item The underlying holomorphic parabolic bundle $E_\bullet$ is stable parabolic.
\item There exists a holomorphic subbundle $F$ of $E$ such that $F_\bullet$ is 
stable parabolic and $(E, \rst{E})$ is isomorphic to 
$F_\bullet\oplus (\sigma_X^*\overline{F})_\bullet$ as real (resp. quaternionic) 
parabolic bundle.
\end{enumerate}
\end{proposition}
\begin{proof}
If the underlying holomorphic parabolic bundle $E_\bullet$ is not stable parabolic, then there exists a non-zero 
subbundle $F$ of $E$ such that $\mathrm{p}\mu(F) \geq \mathrm{p}\mu(E)$. By Proposition 
\ref{prop-rsp-hsp}, $E_\bullet$ is parabolic semistable, and hence we have $\mathrm{p}\mu(F) = \mathrm{p}\mu(E)$.
In particular, $F_\bullet$ is parabolic semistable. 
Let $E'$ be the subbundle generated by $\rst{E}$-invariant subsheaf
$F\cap \sigma_X^*\overline{F}$ of $E$, and $E''$ be the subbundle generated by 
$\rst{E}$-invariant subsheaf $F + \sigma_X^*\overline{F}$ of $E$. 

Consider the short exact sequence 
$$
0\ra E'_\bullet \ra F_\bullet \oplus (\sigma_X^*\overline{F})_\bullet\ra 
E''_\bullet\ra 0
$$
of parabolic bundles, where the map 
$F_\bullet \oplus (\sigma_X^*\overline{F})_\bullet\ra E''_\bullet$
is a morphism of real parabolic bundles, where $F \oplus (\sigma_X^*\overline{F})$ is 
endowed with real structure $\tilde{\sigma}^+$ (resp. quaternionic structure $\tilde{\sigma}^-$)
(see \cite[page 7]{Sch12}). Since $E'$ and $E''$ are $\sigma^E$-invariant
subbundles of $E$, and $(E, \rst{E})$ is real stable parabolic bundle, we have
$$
\frac{\mathrm{deg}(E') + \mathrm{wt}(E')}{\mathrm{rank}(E')} = 
\mathrm{p}\mu(E') < \mathrm{p}\mu(E) = \mathrm{p}\mu(F) 
$$
and
$$
\frac{\mathrm{deg}(E'') + \mathrm{wt}(E'')}{\mathrm{rank}(E'')} = 
\mathrm{p}\mu(E'') < \mathrm{p}\mu(E) = \mathrm{p}\mu(F) 
$$
Hence, we have
\begin{equation}\label{prop-rsp-eq-1}
\mathrm{rank}(F)(\mathrm{deg}(E') + \mathrm{wt}(E')) <
\mathrm{rank}(E')(\mathrm{deg}(F) + \mathrm{wt}(F))
\end{equation}
\begin{equation}\label{prop-rsp-eq-2}
\mathrm{rank}(F)(\mathrm{deg}(E'') + \mathrm{wt}(E'')) <
\mathrm{rank}(E'')(\mathrm{deg}(F) + \mathrm{wt}(F))
\end{equation}
Note that $\mathrm{deg}(E') + \mathrm{deg}(E'') = 2 \mathrm{deg}(F)$
and $\mathrm{rank}(E') + \mathrm{rank}(E'') = 2 \mathrm{rank}(F)$.
Using this, from \eqref{prop-rsp-eq-1} and \eqref{prop-rsp-eq-2}, we have
$\mathrm{wt}(E') + \mathrm{wt}(E'') < \mathrm{wt}(F)$, which is a contradiction.
Hence, $E' = 0$ and $E'' = E$ (if $E''$ is a proper subbundle of $E$, then
we will have $\mathrm{p}\mu(E'') < \mathrm{p}\mu(F)$. From this, one has
$$
\frac{\mathrm{deg}(F) + \mathrm{wt}(F)}{\mathrm{rank}(F)} =
\frac{\mathrm{deg}(E'') + \mathrm{wt}(E'')}{\mathrm{rank}(E'')} <
\frac{\mathrm{deg}(F) + \mathrm{wt}(F)}{\mathrm{rank}(F)}
$$
i.e. $\mathrm{deg}(F) + \mathrm{wt}(F) < \mathrm{deg}(F) + \mathrm{wt}(F)$, 
a contradiction).
\end{proof}

From the above result, we can deduce the following result that real stability implies simplicity in the category
of real semistable parabolic bundles.

\begin{cor}\label{cor-rs-simple}
Let $(E, \rst{E})$ be a real stable parabolic bundles on $(X, \sigma_X)$. 
\begin{enumerate}
\item If the underlying holomorphic bundle $E_\bullet$ is parabolic stable, then the set of real parabolic endomorphism
of $(E, \rst{E})$ is
$$
(\mathrm{ParEnd}(E_\bullet))^{\rst{E}} = \{\lambda \;|\; \lambda \mathrm{Id}_E \in \R\} \cong_\R \R\,.
$$
\item If $(E, \rst{E})$ is isomorphic to $F_\bullet\oplus (\overline{\sigma_X^*F})_\bullet$ as real parabolic bundle, 
then 
$$
(\mathrm{ParEnd}(E_\bullet))^{\rst{E}} = \{(\lambda, \overline{\lambda})\;|\; \lambda\in \C\} \cong_\R \C.
$$
\end{enumerate}
\end{cor}
\begin{proof}
If $E_\bullet$ is parabolic stable, then it is known that 
$$
\mathrm{ParEnd}(E_\bullet) = \{\lambda \;|\; \lambda \mathrm{Id}_E \in \C\} \cong \C.
$$
The induce real structure on $\mathrm{ParEnd}$ is given by $\lambda \mapsto \overline{\lambda}$, and hence
$$
(\mathrm{ParEnd}(E_\bullet))^{\rst{E}} = \{\lambda \;|\; \lambda \mathrm{Id}_E \in \R\} \cong_\R \R\,.
$$
The proof of (2) follows in the same way as that of \cite[page 9]{Sch12} using the fact that the homothety gives the 
parabolic endomorphism of a stable parabolic bundle.
\end{proof}

\subsection*{Jordan-H\"older filtrations}
In this section, we study Jordan-H\"older filtrations of real (resp. quaternionic) semistable 
parabolic bundles of fixed type $\tau_p$. 

If $E_\bullet$ is a holomorphic semistable parabolic bundle, then there exists a filtration
$$
0 = E_0 \subset E_1 \subset \cdots \subset E_\ell = E
$$
such that for each $i = 1, 2, \dots, \ell$, the parabolic quotient bundle $(E_i/E_{i-1})_\bullet$
is stable with $\mathrm{p}\mu(E_i/E_{i-1}) = \mathrm{p}\mu(E)$. 
Such a filtration is called a Jordan-H\"older filtration of $E_\bullet$, which generally may not be unique. 
However, the associated graded object 
$$
\mathrm{gr}(E_\bullet) := \bigoplus_{i=1}^\ell (E_i/E_{i-1})_\bullet
$$
is unique up to isomorphism. A holomorphic parabolic vector bundle, which is a direct sum of stable
parabolic bundles of the equal parabolic slope, is called a poly-stable parabolic bundle. Note that the associated 
graded  $\mathrm{gr}(E_\bullet)$ is a poly-stable parabolic bundle.

We say that two semistable parabolic bundles $E_\bullet$ and $F_\bullet$ are $S$-equivalent if
the associated graded objects $\mathrm{gr}(E_\bullet)$ and $\mathrm{gr}(F_\bullet)$ are isomorphic
as parabolic bundles. The isomorphism class of an associated graded object of $E_\bullet$ is called the
$S$-equivalence class of $E_\bullet$.

\begin{definition}\label{def-real-polystable-parabolic}\rm{
A real (resp. quaternionic) parabolic bundle $(E, \rst{E})$ on $(X, \sigma_X)$ is 
called real (resp. quaternionic) polystable if there exists real (resp. quaternionic) stable 
parabolic bundles $\{(F_i, \rst{F_i})\}_{i = 1, 2, \dots, k}$ of equal parabolic 
slope such that $\rst{E} = \rst{F_1}\oplus \cdots \oplus \rst{F_k}$ and
$$
E_\bullet \cong \bigoplus_{i=1}^k (F_i)_\bullet\;.
$$
}
\end{definition}

\begin{theorem}\label{thm-RP-abelian}
Let $\catRP$ \rm{(}resp. $\catQP$\rm{)} denote the category of real (resp. quaternionic) semistable parabolic 
bundles on $(X, \sigma_X)$ having fixed parabolic type $\tau_p$. Then, $\catRP$ \rm{(}resp. $\catQP$ \rm{)} is 
an abelian category. Moreover, the simple objects in $\catRP$ 
are precisely the real (resp. quaternionic) stable parabolic bundles having parabolic type $\tau_p$.
\end{theorem}
\begin{proof}
Let $\catP$ be the category of semistable holomorphic parabolic bundles on $X$ having fixed parabolic type $\tau_p$.
By Proposition \ref{prop-rsp-hsp}, the category $\catRP$ is a strict subcategory of $\catP$. Since $\catP$
is an abelian category, we only need to check that if 
$\varphi \colon (E, \rst{E})\ra (F, \rst{F})$
is morphism in $\catRP$, then $\mathrm{Ker}(\varphi)$ and $\mathrm{Im}(\varphi)$ are real vector bundles. 
Since $\varphi$ is a morphism of real vector bundles, we have 
$\varphi \circ \rst{E} = \rst{F}\circ \varphi$. 
From this, it follows that $\mathrm{Ker}(\varphi)$ is $\rst{E}$-invariant and $\mathrm{Im}(\varphi)$
is $\rst{F}$-invariant.

Let $(E, \rst{E})$ be a real stable parabolic bundle having parabolic type $\tau_p$. If  
$(E, \rst{E})$ admit a non-trivial subobject, say $(E', \rst{E'})$ in $\catRP$,
then it gives a contradiction to the fact that $(E, \rst{E})$ be a real stable parabolic bundle.
Hence, $(E, \rst{E})$ does not admit a non-trivial subobject in $\catRP$. This implies that 
$(E, \rst{E})$ is a simple object in $\catRP$. Conversely, if $(E, \rst{E})$
is a simple object $\catRP$, then for any non-trivial real subbundle $F$ of $E$, we have
$\mathrm{p}\mu(F) < \mathrm{p}\mu(E)$. This completes the proof.
\end{proof}

\begin{definition}\rm{
Let $(E, \rst{E})$ be a real semistable parabolic bundles on $(X, \sigma_X)$. 
By a \emph{real (resp. quaternionic) parabolic Jordan-H\"older filtration} of $(E, \rst{E})$, we 
mean a filtration
$$
0 = E_0 \subset E_1 \subset \cdots \subset E_\ell = E
$$
by $\sigma^E$-invariant subbundles of $E$ such that for each $i = 1, 2, \dots, \ell$, the 
quotient real (resp. quaternionic) parabolic bundle $(E_i/E_{i-1}, \tilde{\sigma}_i)$
is real (resp. quaternionic) stable with $\mathrm{p}\mu(E_i/E_{i-1}) = \mathrm{p}\mu(E)$. 
}
\end{definition}

\begin{proposition}\label{real-JH-filtration}
Every real (resp. quaternionic) semistable parabolic bundle $(E, \rst{E})$ admits a real 
Jordan-H\"older filtration.
\end{proposition}
\begin{proof}
Since the category $\catRP$ is an abelian, Noetherian, and Artinian, the Jordan-H\"older
theorem holds in $\catRP$.
\end{proof}

\begin{cor}\label{cor--real-S-equiv}
The holomorphic $S$-equivalence class of a real (resp. quaternionic) semistable parabolic bundle 
$(E, \rst{E})$ contains a real (resp. quaternionic) polystable parabolic bundle. 
Any two such objects are isomorphic as real (resp. quaternionic) polystable parabolic bundles.
\end{cor}
\begin{proof}
The first assertion follows from Propositions \ref{real-JH-filtration} and \ref{prop-rsp} 
(see Definition \ref{def-real-polystable-parabolic}). For the second assertion, it is enough to 
show that two real (resp. quaternionic) stable parabolic bundles $(E_1, \rst{E_1})$ 
and $(E_2, \rst{E_2})$ such that $(E_1)_\bullet \cong (E_2)_\bullet$ 
as holomorphic parabolic bundles are, in fact, isomorphic as real (resp. quaternionic) parabolic bundles. 
By Proposition \ref{prop-rsp}, we need to consider the following two cases to conclude the argument 
using induction.

\noindent \textbf{Case-1:} Suppose that $(E_1)_\bullet$ and $(E_2)_\bullet$ are 
stable holomorphic parabolic bundles.

Let $\varphi\colon (E_1)_\bullet \ra (E_2)_\bullet$ be an isomorphism of holomorphic 
parabolic bundle. By following the similar arguments as in \cite[Proposition 2.8]{Sch12}, we can 
conclude that $(E_1, \rst{E_1})$ and $(E_2, \rst{E_2})$ are 
isomorphic as real (resp. quaternionic) parabolic bundles.

\noindent \textbf{Case-2:} Suppose that 
$(E_1, \rst{E_1})\cong  (F_1)_\bullet\oplus (\sigma_X^*\overline{F_1})_\bullet$ and
$(E_2, \rst{E_2})\cong  (F_2)_\bullet\oplus (\sigma_X^*\overline{F_2})_\bullet$, where
$(F_i)_\bullet$ is stable holomorphic bundle, $i = 1, 2$.

Since $(E_1)_\bullet$ and $(E_2)_\bullet$ are isomorphic as holomorphic polystable 
parabolic bundles, it follows that either $(F_1)_\bullet\cong (F_2)_\bullet$ or 
$(F_1)_\bullet\cong (\sigma_X^*\overline{F_2})_\bullet$. From this, it follows that 
the isomorphism $\varphi\colon (E_1)_\bullet \ra (E_2)_\bullet$ of holomorphic 
parabolic bundles is an isomorphism of real (resp. quaternionic) parabolic bundles.
\end{proof}

\section{Gauge theoretic approach to parabolic bundles over Klein surfaces}\label{sec-4}
In this section, we study the induced real structure on the space of Sobolev connections and the 
gauge group, which respect the parabolic structure on the fixed smooth real (resp. quaternionic)
parabolic bundle $(E_\bullet, \sigma^E)$.

\subsection{Real structure on the space of Sobolev connections}
Now let us fix a real (resp. quaternionic) smooth bundle $(E, \sigma^E)$ on $(X, \sigma_X)$ with 
a real parabolic structure over $S$, where $S$ is a finite subset of  $X$ such that $\sigma_X(S) = S$. 
Let $\phi\colon \sigma_X^*\overline{E}\ra E$ be the isomorphism, determined by the real (resp. quaternionic) 
structure of $E$, such that $\sigma_X^*\overline{\phi}\circ \phi = 
\mathrm{Id}_E~(\mathrm{resp.} - \mathrm{Id}_E)$.
Note that $A^s(E)$ and $A^{q, s}(E)$ have induced real structure from the real structure on $E$, which we shall denote by simply $\tilde{\sigma}$ and the induced isormorphism by $\tilde{\phi}$.
For $\parE \in \mathcal{C}$, we define $\alpha_\sigma(\parE)\colon A^0(E)\ra A^{0,1}(E)$
as follows: For any $s\in A^0(E)$
$$
\alpha_\sigma(\parE)(s)= \tilde{\phi}(\overline{\partial}_{\sigma_X^*\overline{E}}(\phi^{-1}s))\,.
$$
It is clear that $\alpha_\sigma^2 = \mathrm{Id}_\mathcal{C}$. There is also an involution $\gamma_\sigma\colon \mathcal{G}_{\mathrm{par}} \ra \mathcal{G}_{\mathrm{par}}$ given by 
$
g\mapsto \phi(\sigma_X^*\bar{g})\phi^{-1}\;.
$
As usual, $\mathcal{G}_{\mathrm{par}}$ acts on $\mathcal{C}$ as $g\cdot \parE := g(\parE g^{-1})$. 

\begin{lemma}
For $g\in \mathcal{G}_{\mathrm{par}}$ and $\parE \in \mathcal{C}$, we have
$
\alpha_\sigma(g\cdot \parE) = \gamma_\sigma(g)\cdot \alpha_\sigma(\parE).
$
\end{lemma}
\begin{proof}
Let $g\in \mathcal{G}_{\mathrm{par}}$ and $\parE \in \mathcal{C}$. Then,
\begin{align*}
\alpha_\sigma(g.\overline{\partial}_E)
& =\tilde{\phi} \circ \sigma_X^*\overline{g}(\overline{\partial}_{\sigma^*_X \overline{E}} \sigma_X^*\overline{g}^{-1})\circ\phi^{-1}\\
& =(\tilde{\phi}\circ \sigma_X^*\overline{g} \circ \tilde{\phi}^{-1})\circ \tilde{\phi} \circ (\overline{\partial}_{\sigma^*_X\overline{E}}(\phi^{-1}(\phi\circ \sigma_X^*\overline{g}^{-1}\circ \phi^{-1})))\\
& =\gamma_\sigma(g)(\alpha_\sigma(\overline{\partial}_E)\gamma_\sigma(g)^{-1})\\
& =\gamma_\sigma(g).\alpha_\sigma(\overline{\partial}_E)
\end{align*}
\end{proof}

Let $\mathcal{C}^{\alpha_\sigma} = \{\parE \in \mathcal{C}\;|\; \alpha_\sigma(\parE) = \parE\}$ and
$\mathcal{G}_{\mathrm{par}}^{\gamma_\sigma} = \{g\in \mathcal{G}_{\mathrm{par}}\;|\; \gamma_\sigma(g) = g\}$. Then, 
the subgroup $\mathcal{G}_{\mathrm{par}}^{\gamma_\sigma}$ acts on the space $\mathcal{C}^{\alpha_\sigma}$.
The orbit space
$\mathcal{C}^{\alpha_\sigma}/\mathcal{G}_{\mathrm{par}}^{\gamma_\sigma}$ is in bijection with the set of 
isomorphism classes of real (resp. quaternionic) parabolic bundles whose underlying smooth real 
(resp. quaternionic) parabolic bundles are smoothly isomorphic to $(E_\bullet, \sigma^E)$.

Let us fix an adapted Hermitian metric $h$ on $(E_\bullet, \sigma^E)$. For $D \in \mathcal{A}^p$, 
we define $\alpha_\sigma(D)$ as follows:
$$
d_{\alpha_\sigma(D)}:=\tilde{\phi}\circ d_{\sigma_X^*\overline{D}}\circ \phi^{-1}  \;,
$$
where $\sigma_X^*\overline{D}$ is the induced connection on $\sigma_X^*\overline{E}$ and 
$\tilde{\phi}:L^p_1(A^1(\sigma_X^*\overline{E}))\ra L^p_1(A^1(E))$ is isomorphism induced by the
real structure on $T^*X\otimes E$.

Also, the space $\mathcal{A}^p$ is an affine space with the group of translations 
$L^p_1(A^1(\mathfrak{u}(E)))$. Since $\mathfrak{u}(E)$ is compact Lie algebra, it admits and 
Ad-invariant non-degenerate symmetric bilinear form 
$\langle\;,\;\rangle:\mathfrak{u}(E)\times \mathfrak{u}(E)\rightarrow \mathbb{R}$ and the wedge 
product $\wedge:A^1\times A^1\rightarrow A^2$ is skew-symmetric. By composing the two maps, we have 
a skew-symmetric bilinear form $\omega$ given by
$$
L^p_1(A^1(\mathfrak{u}(E)))\times L^p_1(A^1(\mathfrak{u}(E)))\rightarrow L^p_1(A^2(\mathfrak{u}(E)\otimes \mathfrak{u}(E)))\rightarrow \mathbb{R}
$$
$$ (a,b)\mapsto a\wedge b \mapsto \displaystyle\int_X\langle a\wedge b\rangle
$$
which is non-degenerate. This skew-symmetric, non-degenerate, bilinear form $\omega$ is called 
Atiyah-Bott symplectic form.

\begin{proposition} Let $p\in (1, 2)$ be such that the condition \eqref{eq-1} holds. For $D \in \mathcal{A}^p$, 
we have $\alpha_\sigma(D) \in \mathcal{A}^p$. The map $\alpha_\sigma:\mathcal{A}^p\rightarrow \mathcal{A}^p$ 
given by, $A\mapsto \alpha_\sigma(D) = \tilde{\phi}\sigma_X^* \overline{A}\phi^{-1}$ is an anti-symplectic isometric involution.
\end{proposition} 
\begin{proof}
Let $(e_i)_{i=1}^r$ and $(f_i)_{i=1}^r$ be local frames which respect flag structure of $E$ and 
$\sigma_X^*\overline{E}$ at $x\in S$. Let $(\phi)^i_j$ be the matrix of $\phi$ (with respect to 
frames $(e_i)$ and $(f_i)$) which respects flag structure, i. e.,
$
(\phi)^i_j=0 \text{ if }\alpha_i<\alpha_j\;.
$
The matrix of $\phi$ with respect to adapted frames $(\frac{e_i}{|z|^{\alpha_i}})$ and 
$(\frac{f_i}{|z|^{\alpha_j}})$ is $|z|^{\alpha_i-\alpha_j}(\phi)^i_j$, which is in $D^p_2(\text{End E})$, 
since $p$ satisfies the condition \eqref{eq-1}.
Note that $\{d\overline{z_i\circ \sigma_X}\otimes f_j\}_{j=1}^r$ is a local frame for 
$T^*X\otimes \sigma_X^*\overline{E}$ around $x$, and the matrix of $\tilde{\phi}$ also respects 
flag structure and lies in $D^p_2(A^1(\mathrm{End} (E))$.

If $(e_i)_{i=1}^r$ is a local frame for $E$ on chart $(U,z)$ around $x$, then 
$(\phi^{-1}(e_i))_{i=1}^r$ will be a local frame for $\sigma_X^*\overline{E}$ on chart $(U,z)$, 
and also a local frame for $E$ on chart $(\sigma_X(U),\overline{z\circ \sigma_X})$.

Any connection $D\in \mathcal{A}^p$ can be expressed (locally on chart $(U,z)$) as
$$
d_D=d+\frac{\alpha}{2}\left(\frac{dz}{z}-\frac{d\overline{z}}{\overline{z}}\right)+
dz\otimes a^{(1,0)}+d\overline{z}\otimes a^{(0,1)}\;,
$$
where $a^{(1,0)},a^{(0,1)}\in D^P_1(\text{End E})$. Similarly, a connection $D$ can be expressed 
locally on chart $(\sigma_X(U),\overline{z\circ \sigma_X})$ as
$$
d_D=d+\frac{\alpha}{2}\left(\frac{d(\overline{z\circ \sigma_X})}{(\overline{z\circ \sigma_X})}-
\frac{d(z\circ \sigma_X)}{(z\circ \sigma_X)}\right)+d(\overline{z\circ\sigma_X})\otimes b^{(1,0)}+
d(z\circ \sigma_X)\otimes b^{(0,1)}\;,
$$ 
where $b^{(1,0)},b^{(0,1)}\in D^P_1(\text{End E})(\sigma_X(U))$.

Hence, the induced connection $\sigma_X^*\overline{D}$ can be expressed locally on chart $(U,z)$ as
$$
d_{\sigma_X^*\overline{D}}=d+\frac{\alpha}{2}\left(\frac{d(\overline{z\circ \sigma_X})}{(\overline{z\circ \sigma_X})}
-\frac{d(z\circ \sigma_X)}{(z\circ \sigma_X)}\right)+
d(\overline{z\circ\sigma_X})\otimes b^{(1,0)}+d(z\circ \sigma_X)\otimes b^{(0,1)}
$$ 
and on chart $(\sigma_X(U),\overline{z\circ \sigma_X})$ as
$$
d_{\sigma_X^*\overline{D}}=d+\frac{\alpha}{2}\left(\frac{dz}{z}-\frac{d\overline{z}}{\overline{z}}\right)+dz\otimes a^{(1,0)}+d\overline{z}\otimes a^{(0,1)}\;.
$$ 

In the chart $(U,z)$, we have 
\[
\begin{array}{ll}
d_{\alpha_{\sigma}(D)}(e_i) 
& =\tilde{\phi}\circ d_{\sigma_X^*\overline{D}}\circ (\phi^{-1}(e_i)) \\
& \\
& =d+\tilde{\phi}\circ \Big[\frac{\alpha_{ii}}{2}\left(\frac{d(\overline{z\circ \sigma_X})}{(\overline{z\circ\sigma_X})}-\frac{d(z\circ\sigma_X)}{(z\circ\sigma_X)}\right)\phi^{-1}(e_i)+d(\overline{z\circ\sigma_X})\otimes \displaystyle \sum_{j}[b^{(1,0)}]^{j}_i\phi^{-1}(e_{j})\\
& \quad \quad +d(z\circ\sigma_X)\otimes \displaystyle \sum_{j}[b^{(0,1)}]^{j}_i\phi^{-1}(e_{j})\Big]\\
& \\
& =d+\frac{\alpha_{ii}}{2}(\frac{dz}{z}-\frac{d\overline{z}}{\overline{z}})e_i+dz\otimes \displaystyle \sum_{j}[\overline{b^{(1,0)}\circ \sigma_X}]^{j}_i e_{j}+d\overline{z}\otimes \displaystyle \sum_{j}[\overline{b^{(0,1)}\circ \sigma_X}]^{j}_i e_{j}
\end{array}
\]
Hence, 
$$
d_{\alpha_{\sigma}(D)}\equiv d+\frac{\alpha}{2}\left(\frac{dz}{z}-\frac{d\overline{z}}{\overline{z}}\right)+dz\otimes \overline{b^{(1,0)}\circ \sigma_X}+d\overline{z}\otimes \overline{b^{(0,1)}\circ \sigma_X}.
$$
This shows that $\alpha_\sigma(D)\in \mathcal{A}^p$. Now, in the chart $(\sigma_X(U),\overline{z\circ\sigma_X})$, 
we have
$$
d_{\alpha_{\sigma}(D)}\equiv d+\frac{\alpha}{2}\left(\frac{d(\overline{z\circ\sigma_X})}{(\overline{z\circ\sigma_X})}-\frac{d(z\circ\sigma_X)}{(z\circ\sigma_X)}\right)+d(\overline{z\circ\sigma_X})\otimes \overline{a^{(1,0)}\circ \sigma_X}+d(z\circ\sigma_X)\otimes \overline{a^{(0,1)}\circ \sigma_X}.
$$

The induced connection $\sigma_X^*\overline{\alpha_\sigma(A)}$ can be expressed locally on chart $(U,z)$ as
$$
d_{\sigma_X^*{\alpha_{\sigma}(D)}}\equiv d+\frac{\alpha}{2}\left(\frac{d(\overline{z\circ\sigma_X})}{(\overline{z\circ\sigma_X})}-\frac{d(z\circ\sigma_X)}{(z\circ\sigma_X)}\right)+d(\overline{z\circ\sigma_X})\otimes \overline{a^{(1,0)}\circ \sigma_X}+d(z\circ\sigma_X)\otimes \overline{a^{(0,1)}\circ \sigma_X}.
$$
Hence, in chart $(U,z)$, we have
\begin{align*}
d_{\alpha_{\sigma}(\alpha_{\sigma}(D))}(e_i) 
&= \tilde{\phi}\circ d_{\sigma^*_X\overline{(\alpha_\sigma(A))}}\circ(\phi^{-1}(e_i))\\
&=d+\tilde{\phi}\circ\Bigl[\frac{\alpha_{ii}}{2}\left(\frac{d(\overline{z\circ \sigma_X}}{\overline{z\circ \sigma_X}}-\frac{d(z\circ\sigma_X)}{z\circ \sigma_X}\right)\phi^{-1}(e_i)\\
&\qquad +d(\overline{z\circ\sigma_X})\otimes \displaystyle\sum_j[\overline{a^{(1,0)}\circ\sigma_X}]^j_i\phi^{-1}(e_j)\\
&\qquad +d(z\circ \sigma_X)\otimes \displaystyle\sum_j[\overline{a^{(0,1)}\circ\sigma_X}]^j_i\phi^{-1}(e_j)\Bigr]\\
&=d+\frac{\alpha_{ii}}{2}\left(\frac{dz}{z}-\frac{d\overline{z}}{\overline{z}}\right)e_i+d(z)\otimes \displaystyle\sum_j[a^{(1,0)}]^j_ie_j\\
&\qquad+d(\overline{z})\otimes \displaystyle\sum_j[a^{(0,1)}]^j_ie_j \\
&=d_D(e_i)
\end{align*}
From this, it follows that the map $\alpha_\sigma \colon \mathcal{A}^p \ra \mathcal{A}^p$ is an involution.

The map $L^p_1(A^1(X,\mathfrak{u}(E)))\ra L^p_1(A^1(X,\mathfrak{u}(E)))$ given by 
$a\mapsto \alpha_\sigma(a):=\phi\sigma^*\overline{a}\phi^{-1}$
is anti-linear. Since $\langle\; ,\;\rangle$ is real valued on anti-Hermitian matrices and $\sigma$ is an orientation reversing isometry of $X$, we have
$$
\omega(\alpha_\sigma(a),\alpha_\sigma(b))=-\omega(a,b)\;.
$$
\end{proof}

We say that $D\in \mathcal{A}^p$ is real (resp. quaternionic) if $\alpha_\sigma(D) = D$. Let $\gamma_\sigma$ denote the induced 
involution on $\mathcal{G}^p$.
We denote by $\beta_\sigma$ the involution
on $L^p(A^2(\mathfrak{u}(E))$ given by
$$
\beta_\sigma(\omega):= \tilde{\phi}\circ \sigma_X^*\overline{\omega}\circ \phi^{-1}\;,
$$
where $\tilde{\phi}$ is the isomorphism induced by the real structure on the bundle 
$\Lambda^2T^*X\otimes \mathfrak{u}(E))$ (see \eqref{real}).

The following result is analogous to that of \cite[Proposition 3.5]{Sch12}.

\begin{proposition}\label{prop-1}With the above notations:
\begin{enumerate}
\item For $g\in \mathcal{G}^p$ and $D\in \mathcal{A}^p$, we have $\alpha_\sigma(g(D)) = 
\gamma_\sigma(g)(\alpha_\sigma(D))$.
\item For $D\in \mathcal{A}^p$, we have $F_{\alpha_\sigma(D)} = \beta_\sigma(F_D)$.
\end{enumerate}
\end{proposition}
\begin{proof}
For $g\in \mathcal{G}^p$ and $D\in \mathcal{A}^p$ we have,
\[
\begin{array}{ll}
\alpha_\sigma(g\cdot D) 
& = \tilde{\phi} \circ \sigma_X^*(g\cdot D) \circ \phi^{-1} \\
& = \tilde{\phi}\circ (\sigma_X^*D+(d_{\sigma^*_X D}\sigma_X^*g)(\sigma_X^*g^{-1}))\circ \phi^{-1} \\
& = \alpha_\sigma(D) + ((\tilde{\phi}\circ d_{\sigma^*_X D} \circ \phi^{-1}) (\phi \circ \sigma_X^*g \circ \phi^{-1}))((\phi \circ \sigma_X^*g^{-1} \circ \phi^{-1}))\\
& = \alpha_\sigma(D) + (d_{\alpha_\sigma(D)}\gamma_\sigma(g))\gamma_\sigma(g^{-1}) \\
& =  \gamma_\sigma(g)\cdot \alpha_\sigma(D)\,.
\end{array}
\]
For a section $s\in A^0(E)$, we have
\[
\begin{array}{ll}
d_{\alpha_{\sigma}(D)}(s)
& =\tilde{\phi}\circ d_{\sigma^*\overline{D}}\circ \phi^{-1}(s) \\
& =\tilde{\phi}\circ \overline{d_D}(\overline{\phi^{-1}}\circ s\circ \sigma_X)\circ \sigma_X
\end{array}
\]
Hence,  
\[
\begin{array}{ll}
d_{\alpha_{\sigma}(D)}\circ d_{\alpha_{\sigma}(D)}(s)

& = \tilde{\phi}\circ (\overline{d_D}\circ d_D(\overline{\phi^{-1}}\circ s\circ \sigma_X)\circ \sigma_X \\

& = \tilde{\phi}\circ \sigma_X^*(d_D\circ d_D)\circ \phi^{-1}(s) \\
&=\tilde{\phi}\circ F_{\sigma^*_X \overline{D}}\phi^{-1}(s)
\end{array}
\]
From this, we can conclude that $F_{\alpha_{\sigma}(D)}\equiv \beta_{\sigma}(F_D)$.
\end{proof}

Let $\mathcal{A}^p_\mathrm{ss}:= (\star F)^{-1}(2\pi \sqrt{-1}\mathrm{par}\text{-}\mu(E))$. 
From the above Proposition \ref{prop-1}, it follows that the involution $\alpha_\sigma$ induces an involution 
on $\mathcal{A}^p_\mathrm{ss}$. Moreover, the group $\mathcal{G}^{p, \sigma}$ 
acts on the fixed point set
$
\mathcal{A}^{p,\alpha_\sigma}_\mathrm{ss},
$
of the involtion $\alpha_\sigma$.
For a real connection $D\in \mathcal{A}^p$, we denote by $O_{\mathcal{G}^p}(D)$ the orbit of $D$ 
with respect to the action of $\mathcal{G}^p$ in $\mathcal{A}^p$, and by $O_{\mathcal{G}^{p, \sigma}}(D)$ the orbit 
of $D$ with respect to the action of $\mathcal{G}^{p, \sigma}$ in $\mathcal{A}^{p, \alpha_\sigma}$.

\begin{proposition}\cite[Proposition 3.6]{Sch12}\label{prop-2}\rm{
If $D$ is a real connection in $\mathcal{A}^p$, which defines a poly-stable real (resp. quaternionic) structure, then 
$O_{\mathcal{G}^p}(D) \cap \mathcal{A}^{p, \alpha_\sigma} = O_{\mathcal{G}^{p, \sigma}}(D)$.
}
\end{proposition}
\begin{proof}
The proof follows in the same line of arguments as in \cite[Proposition 3.6]{Sch12} using Proposition \ref{prop-rsp}
and Biquard's Theorem \ref{Biquard-thm}.
\end{proof}

\begin{theorem}\label{main-thm}
Let $(E_\bullet, \sigma^E)$ be a real (resp. quaternionic) smooth parabolic bundle over $(X, \sigma)$ 
having parabolic type $\tau_p$. Let $h$ be an adapted Hermitian metric $h$ on $E_\bullet$. 
Let $\mathcal{N}^{\tau_p}_{\tilde{\sigma}}$ denote the Lagrangian quotient 
$\mathcal{A}^{p,\alpha_\sigma}_\mathrm{ss}/\mathcal{G}^{p, \sigma}$. Then, the points of the space 
$\mathcal{N}^{\tau_p}_{\tilde{\sigma}}$ are in bijection with the real (resp. quaternionic) $S$-equivalence 
classes of real (resp. quaternionic) semistable parabolic vector bundles that are smoothly isomorphic 
to $(E_\bullet, \sigma^E)$.
\end{theorem}
\begin{proof}
The proof follows in the same line of arguments as in the proof of \cite[Theorem 3.7]{Sch12} with the aid of 
the Theorem \ref{Biquard-thm}, Proposition \ref{prop-1}, Proposition \ref{prop-2}.
\end{proof}

\begin{remark}\rm{
From the above Theorem \ref{main-thm}, it follows that a real (resp. quaternionic) parabolic bundle $(E_\bullet, \sigma^E)$
is polystable if and only if it admits a real (resp. quaternionic) adapted Hermitian-Yang-Mills connection 
(see \cite[Theorem 3.6]{BS20}).
}
\end{remark}

For a connection $D\in \mathcal{A}$, let us consider the connection $B = \frac{1}{2}(D + \alpha_\sigma(D))$. 
Then, we have
$$F_B = \frac{1}{2}(F_D + F_{\alpha_\sigma(D)})\,.$$
To see this,
let $\{e_i\}_{i=1}^r$ be local frame of $E$ over $(U,z)$ then $\{\phi^{-1}(e_i)\}_{i=1}^r$ will be local frame of $E$ over $(\sigma_X(U),\overline{z\circ \sigma_X})$. Let $\omega_D=a^{(1,0)}dz+a^{(0,1)}d\overline{z}$, where 
$$
D(e_j)=\displaystyle\sum_i\Bigl\{[a^{(1,0)}]^i_jdz+[a^{(0,1)}]^i_jd\overline{z}\Bigr\}e_i
$$ and $\omega_{\alpha^{\sigma}(D)}=\overline{b^{(1,0)}\circ\sigma_X}dz+\overline{b^{(0,1)}\circ\sigma_X}d\overline{z}$,
where $$D(\phi^{-1}(e_j))=\sum\Bigl\{[b^{(1,0)}]^i_jd(\overline{z_i\circ\sigma_X})+[b^{(0,1)}]^i_jd(z_i\circ\sigma_X)\Bigr\}$$

i.e., $\omega_{\alpha^{\sigma}(D)}=b'^{(1,0)}dz+b'^{(0,1)}d\overline{z}$

Note that 
$$
\begin{array}{ll}
\omega_B
& = [\frac{\omega_D+\omega_{\alpha_{\sigma}(D)}}{2}] \\
& \\
& = [\frac{a^{(1,0)}+b'^{(1,0)}}{2}]dz+[\frac{a^{(0,1)}+b'^{(0,1)}}{2}]d\overline{z} \\
& \\
& = \frac{ab'^{(1,0)}}{2}dz+\frac{ab'^{(0,1)}}{2}d\overline{z},
\end{array}
$$ 
where $ab'^{(1,0)}=a^{(1,0)}+b'^{(1,0)},ab'^{(0,1)}=a^{(0,1)}+b'^{(0,1)}\in A^0(\text{End E})$

Now,
\[
\begin{array}{ll}
\Omega_D
& = d(\omega_D)+\omega_D\wedge\omega_D \\
& \\
& = [\frac{\partial a^{(1,0)}}{\partial\overline{z}}d\overline{z}\wedge dz 
+ \frac{\partial a^{(0,1)}}{\partial z}dz\wedge d\overline{z}]+[a^{(1,0)}]^2dz\wedge dz \\
& \\
& \quad + a^{(1,0)}\cdot a^{(0,1)}\{dz\wedge d\overline{z}+d\overline{z}+dz\}+
[a^{(0,1)}]^2d\overline{z}\wedge d\overline{z}]
\end{array}
\]

Hence, we have
$$\Omega_D=[\frac{\partial a^{(0,1)}}{\partial z}-
\frac{\partial a^{(1,0)}}{\partial\overline{z}}]dz\wedge d\overline{z}$$
$$\Omega_{\alpha^{\sigma}(D)}=[\frac{\partial b'^{(1,0)}}{\partial z}-
\frac{\partial b'^{(0,1)}}{\partial\overline{z}}]dz\wedge d\overline{z}$$

and hence,
\[
\begin{array}{ll}
\Omega_B
& = [\frac{\partial ab'^{(0,1)}}{\partial z}-
\frac{\partial ab'^{(1,0)}}{\partial \overline{z}}]\frac{dz\wedge d\overline{z}}{2} \\
& \\
& = [\{\frac{\partial a^{(0,1)}}{\partial z}-\frac{\partial a^{(1,0)}}{\partial\overline{z}}\}+
\{\frac{\partial b'^{(1,0)}}{\partial z}-
\frac{\partial b'^{(0,1)}}{\partial\overline{z}}\}]\frac{dz\wedge d\overline{z}}{2} \\
& \\
& = \frac{1}{2}[\Omega_D+\Omega_{\alpha_{\sigma}(D)}]
\end{array}
\]
Let $D\in \mathcal{A}$ be such that $\star F_D = -2\pi\sqrt{-1}\mathrm{par}\mu(E)\,.$
Consider the connection $B = \frac{1}{2}(D + \alpha_\sigma(D))$. Then, clearly $\alpha_\sigma(B) = B$. 
From the above computation, it follows that $\star F_B = -2\pi\sqrt{-1}\mathrm{par}\mu(E)\,.$ 
This discussion shows that if there is a holomorphic structure on $E$ such that the resulting 
holomorphic parabolic bundle $E_\bullet$ is semistable, then on can get a holomorphic structure on 
$E$ which is compatible with the real (resp. quaternionic) structure such that 
$(E_\bullet, d_B)$ is semistable.

\subsection*{Equivariant point of view}
Here, we will briefly outline an equivariant approach to address the question of constructing 
suitable moduli space of real (resp. quaternionic) parabolic bundles (discussed above) using 
the equivariant description of real (resp. quaternionic)
parabolic bundles without specific routine details.

Suppose that the weights $0\leq \alpha_1^x < \dots < \alpha_{k_x}^x $ are rational numbers. Let $N$
be a positive integer such that all the weights are integral multiple of $1/N$. Let 
$p\colon (Y, \sigma_Y)\ra (X, \sigma_X)$ be an $N$-fold cyclic ramified covering which is 
ramified over each point of $S$ \cite{Am14a}. Let $\Gamma$ be a Galois group of the covering $p$. 
There is an equivalence between the category of real (resp. quaternionic) $\Gamma$-equivariant 
vector bundles over $(Y, \sigma_Y)$ and the category of real (resp. quaternionic) parabolic bundles
over $(X, \sigma_X; S)$ whose weights are integral multiple of $1/N$. 
Let $\mathfrak{B}(\tau)$ be the set of real $S$-equivalence classes of
real $\Gamma$-equivariant semistable bundles over $(Y, \sigma_Y)$ having local type $\tau$ 
(cf. \cite{Se70} for local type). Let $\mathfrak{B}(\tau_p)$ be the set of real (resp. quaternionic) 
$S$-equivalence classes of real parabolic bundles over $(X, \sigma_X; S)$ having parabolic type 
$\tau_p$, where the parabolic type $\tau_p$ is uniquely determine by the local type $\tau$. 
Using \cite[Proposition 5.4]{Am14a}, it is straightforward to check that, under the equivalence 
$\Psi$ of \cite[Theorem 5.3]{Am14a}, there is a bijection between $\mathfrak{B}(\tau)$ and $\mathfrak{B}(\tau_p)$.

Fix a real (resp. quaternionic) smooth $\Gamma$-equivariant bundle $(W, \sigma^W)$ on $(Y, \sigma_Y)$ 
having local type $\tau$.
Let $\mathscr{C}$ denote the space of holomorphic structure on $W$, and let $\mathscr{G}$ be the gauge group of
$W$. A holomorphic structure $\overline{\partial}_W$ in $W$ is called compatible with $\Gamma$-equivariant structure on $W$
if the map $\overline{\partial}_W \colon A^0(W)\ra A^{0, 1}(W)$ is $\Gamma$-equivariant. Let $\mathscr{C}_\Gamma$ be the
set of all holomorphic structures compatible with the $\Gamma$-equivariant structure on $W$. Let
$\mathscr{G}_\Gamma$ be the subgroup of $\mathscr{G}$ consisting of $\Gamma$-equivariant automorphisms of $W$.
There is induced involution on $\mathcal{D}$, which we denote by $\tilde{\alpha}_\sigma$. Similarly, 
we have the induced involution $\tilde{\gamma}_\sigma$ on $\mathcal{H}$.

Let $\mathscr{D}_\Gamma^{\tilde{\alpha_\sigma}} := \{\overline{\partial}_W \in  \mathscr{D}_\Gamma 
\;|\; \tilde{\alpha}_\sigma(\overline{\partial}_W) = \overline{\partial}_W\}$ and 
$\mathscr{G}_\Gamma^{\tilde{\gamma}_\sigma} := \{g\in \mathscr{G}_\Gamma \;|\; \tilde{\gamma}_\sigma(g) = g\}$. 
It can be easily checked that
the orbit space $\mathscr{D}_\Gamma^{\tilde{\alpha_\sigma}}/\mathscr{G}_\Gamma^{\tilde{\gamma}_\sigma}$ is in bijection
with the set of isomorphism classes of real (resp. quaternionic) $\Gamma$-equivariant holomorphic bundles whose underlying 
smooth real (resp. quaternionic) bundles are smoothly isomorphic to $(W, \sigma^W)$ as $\Gamma$-equivariant bundles.

Fix a $\Gamma$-invariant Hermitian metric $h_W$ on $W$. 
Let $\mathscr{A}$ be the set of all $h_W$-unitary connections on $W$, and the set $\mathscr{A}_\Gamma$ 
of all $h_W$-unitary $\Gamma$-equivariant connections on $W$. Let $\mathscr{U}_\Gamma$ denote the 
subgroup of unitary automorphisms of $(W, h_W)$ consisting of unitary $\Gamma$-automorphism of $(W, h_W)$.

Let $\mathscr{A}_{\Gamma, \mathrm{ss}}:= (\star F)^{-1}(2\pi\sqrt{-1}\mu(W)/N)$. Then, one can check that
$$\mathcal{N}^{\tau}_{\tilde{\sigma}}:=\mathscr{A}_{\Gamma, \mathrm{ss}}^{\tilde{\alpha}_\sigma}/\mathscr{U}_\Gamma^{\tilde{\gamma}_\sigma} 
\simeq \mathfrak{B}(\tau) \simeq \mathfrak{B}(\tau_p).
$$
The bijection $\mathcal{N}^{\tau}_{\tilde{\sigma}} \stackrel{\simeq}{\ra} \mathfrak{B}(\tau)$ can be 
proved by establishing the results of \cite{Sch12} in the equivariant set-up. The second bijection
$\mathfrak{B}(\tau) \stackrel{\simeq}{\ra} \mathfrak{B}(\tau_p)$ is a consequence of the preservation
of stability under the equivalence $\Psi$ of \cite[Theorem 5.3]{Am14a} as mentioned above.
In this approach, one can avoid the use of the theory of weighted
Sobolev spaces; while working with the rational weights.

\subsection{Real points of the moduli scheme}
Let $M_X^\mathrm{ss}(\tau_p)$ be the moduli scheme of stable holomorphic parabolic bundles of parabolic 
type $\tau_p$. Then, we get a map $\sigma_M \colon M_X^\mathrm{ss}(\tau_p) \ra M_X^\mathrm{ss}(\tau_p)$
given by $[E_\bullet]\mapsto [\sigma_X^*\overline{E}_\bullet]$ on the closed points.
\begin{proposition}
The map $\sigma_M \colon M_X^\mathrm{ss}(\tau_p) \ra M_X^\mathrm{ss}(\tau_p)$ is a semi-linear involution
of $\C$-schemes.
\end{proposition}
\begin{proof}
Let $T$ be a $\C$-scheme and $E_\bullet$ be a flat family of semistable parabolic bundles of type $\tau_p$
parametrized by $T$. Consider the morphism $\sigma_T:= \sigma_X\times \mathrm{Id}_T \colon X\times_\C T \ra X\times_\C T$. 
Then, $\sigma_T^*\overline{E}_\bullet$ is flat over $T$ and for any $t\in T$, we have 
$\sigma_T^*\overline{E}_{t_\bullet} \cong \sigma_X^*\overline{E_t}_\bullet$ as parabolic bundles over $(X, S)$. Therefore, 
$\sigma_T^*\overline{E}_\bullet$ is a flat family of semistable parabolic bundles of type $\tau_p$. By universal property of
moduli scheme $M_X^\mathrm{ss}(\tau_p)$, the map $T\ra M_X^\mathrm{ss}(\tau_p)$ given by 
$t\mapsto [\sigma_T^*\overline{E}_{t_\bullet}]$ is a morphism. Since $T$ and $E_\bullet$ are arbitrary, and $\sigma_T$
being semi-linear involution, it follows that
the map $\sigma_M \colon M_X^\mathrm{ss}(\tau_p) \ra M_X^\mathrm{ss}(\tau_p)$ is a morphism of schemes such that the 
following diagram 
\[
\xymatrix{
M_X^\mathrm{ss}(\tau_p) \ar[r]^{\sigma_M} \ar[d] & M_X^\mathrm{ss}(\tau_p) \ar[d] \\
\mathrm{Spec}(\C) \ar[r]_{\sigma_\C} & \mathrm{Spec}(\C)
}
\]
commutes. 
\end{proof}

The elements of fixed point set $M_X^\mathrm{ss}(\tau_p)(\R)$ of the involution $\sigma_M$ on $M_X^\mathrm{ss}(\tau_p)(\C)$
may have both (real and quaternionic) structures or may be of neither type (see \cite[\S 2.5]{Sch12} for the discussion in the 
usual case). The situation is better in the case of a (geometrically) stable locus. 

\begin{lemma}\label{lemma-gs-rqs}
Let $E_\bullet$ be a stable holomorphic parabolic bundle on $X$ with $\sigma_X^*\overline{E}_\bullet 
\cong E_\bullet$. Then, $E_\bullet$ is either real or quaternionic, and it can not be both.
\end{lemma}
\begin{proof}
Note that the isomorphism $\varphi \colon E_\bullet \ra \sigma_X^*\overline{E}_\bullet$ is the same as the
anti-holomorphic map $\tilde{\sigma}\colon E\ra E$ which respects the parabolic structure over $S$.
Hence, the composition $\tilde{\sigma}^2$ is a parabolic automorphism of $E_\bullet$. Since $E_\bullet$ is 
stable, we have $\tilde{\sigma}^2 = c\mathrm{Id}_{E}$. The remaining proof follows in the same line as in
\cite{BHH10, Sch12}.
\end{proof}

\begin{lemma}\label{lemma-hiso-riso}
Let $E_\bullet$ be a stable holomorphic parabolic bundle on $X$. If $\tilde{\sigma}$ and $\tilde{\sigma}'$
are two real (resp. quaternionic) structure on $E$ such that $(E_\bullet, \tilde{\sigma})$ and
$(E_\bullet, \tilde{\sigma}')$ are real (resp. quaternionic) parabolic bundles, then $(E_\bullet, \tilde{\sigma}) \cong (E_\bullet, \tilde{\sigma}')$.
\end{lemma}
\begin{proof}
Note that $\tilde{\sigma}\circ \tilde{\sigma}'$ is a parabolic automorphism of $E_\bullet$. Since 
$E_\bullet$ is stable, we have $\tilde{\sigma}\circ \tilde{\sigma}' = \lambda \in \C^*$. As in the proof
of \cite[Proposition 2.8]{Sch12}, we get $\tilde{\sigma} = e^{i\frac{\theta}{2}}\tilde{\sigma}'e^{-i\frac{\theta}{2}}$, where
$\lambda = e^{i\theta}$ for some $\theta \in \R$. This proves that 
$(E_\bullet, \tilde{\sigma}) \cong (E_\bullet, \tilde{\sigma}')$.
\end{proof}

By Proposition \ref{prop-2}, we can see that the map 
$$
\mathcal{N}^{\tau_p}_{\tilde{\sigma}} \ra M_X^\mathrm{ss}(\tau_p)(\C)\;; \quad 
O_{\mathcal{G}^{p, \sigma}}(D) \mapsto O_{\mathcal{G}^p}(D)
$$
is injective. For $D \in \mathcal{A}^{p,\alpha_\sigma}$, we have $\sigma_M(O_{\mathcal{G}^p}(D)) = O_{\mathcal{G}^p}(D)$. 
Hence, it follows that the quotient space $\mathcal{N}^{\tau_p}_{\tilde{\sigma}}$ embeds into the space 
$M_X^\mathrm{ss}(\tau_p)(\R)$ of real points of the moduli scheme $M_X^\mathrm{ss}(\tau_p)$. Let 
$\mathcal{N}^{\tau_p}_{\tilde{\sigma}, s} = \mathcal{N}^{\tau_p}_{\tilde{\sigma}} \cap M_X^\mathrm{s}(\tau_p)(\R)$

For a smooth parabolic bundle $E_\bullet$ with parabolic type $\tau_p$, let $\mathfrak{I}$ denote the 
parabolic gauge conjugacy classes of real or quaternionic structures on $E$. 

\begin{proposition}
$M_X^\mathrm{s}(\tau_p)(\R) = 
\displaystyle \bigsqcup_{[\tilde{\sigma}]\in \mathfrak{I}} \mathcal{N}^{\tau_p}_{\tilde{\sigma}, s}$
\end{proposition}
\begin{proof}
By Theorem \ref{main-thm} and Lemma \ref{lemma-gs-rqs}, we can conclude that $M_X^\mathrm{s}(\tau_p)(\R) = 
\displaystyle \bigcup_{[\tilde{\sigma}]\in \mathfrak{I}} \mathcal{N}^{\tau_p}_{\tilde{\sigma}, s}\,.$ 
If
$[E_\bullet] \in \mathcal{N}^{\tau_p}_{\tilde{\sigma}, s} \cap \mathcal{N}^{\tau_p}_{\tilde{\sigma}', s}$, then
by Lemma \ref{lemma-hiso-riso}, we have $(E_\bullet, \tilde{\sigma}) \cong (E_\bullet, \tilde{\sigma}')$.
Hence, a parabolic gauge transformation conjugates the real structures $\tilde{\sigma}$ and $\tilde{\sigma}'$.
\end{proof}

\begin{remark}\rm{
There is an isomorphism of schemes
$\psi \colon \mathrm{M}_Y^\mathrm{ss}(\tau) \ra M_X^\mathrm{ss}(\tau_p)$ given by 
$[W]\mapsto [(p_*W^\Gamma)_\bullet]$ such that the following diagram
\[
\xymatrix{
\mathrm{M}_Y^\mathrm{ss}(\tau) \ar[r]^\psi \ar[d]_{\varrho_M} & M_X^\mathrm{ss}(\tau_p) \ar[d]^{\sigma_M} \\
\mathrm{M}_Y^\mathrm{ss}(\tau) \ar[r]^\psi &  M_X^\mathrm{ss}(\tau_p)
}
\]
commutes, where $\mathrm{M}_Y^\mathrm{ss}(\tau)$ is the moduli space of $\Gamma$-equivariant semistable vector bundles
on $Y$ having local type $\tau$, and $\varrho_M$ the induced semi-linear involution on $\mathrm{M}_Y^\mathrm{ss}(\tau)$.
Moreover, we have
$$M_X^\mathrm{s}(\tau)(\R) = 
\displaystyle \bigsqcup_{[\tilde{\sigma}]\in \mathfrak{J}} \mathcal{N}^{\tau}_{\tilde{\sigma}, s}$$ where 
$\mathfrak{J}$ denote the gauge conjugacy classes of real or quaternionic structures on $W$, which are 
compatible with the $\Gamma$-equivariant structure on $W$. 
}
\end{remark}

\subsection{Quillen line bundle}
Recall that there is a determinant line bundle $\mathcal{L}$ on $\mathcal{C}$ (cf. \cite{Q85}) such that
the action of $\C^*$ on $\mathcal{L}$ is given by $\lambda\cdot s\mapsto \lambda^{-\chi(E)}s$, where 
$\lambda\in \C^*$ and $\chi(E) = d + r(1-g)$ (see, \cite[p. 49]{BR93}). Fix a point $x\in X\setminus S$.
Consider the line bundle 
$$
\tilde{\mathcal{L}} := \mathcal{L}^r\otimes(\mathrm{det}(\mathcal{C}\times E_x))^{\chi(E)}
$$ 
Then, the action of $\C^*$ on $\tilde{\mathcal{L}}$ is trivial. Note that the quotient map
$\varphi \colon \mathcal{C}_\mathrm{s} \ra \mathrm{M}^s_X(\tau_p)$ is a 
$\mathcal{PG}_\mathrm{par}$-principal bundle,
where $\mathcal{PG}_\mathrm{par} = \mathcal{G}_\mathrm{par}/\C^*$ and 
$$\mathcal{C}_{\mathrm{s}}:= \{\bar{\partial}_E \in \mathcal{C}\;|\; (E_\bullet, \bar{\partial}_E) ~ 
\mbox{is stable parabolic bundle}\}.$$ Hence, the restriction of $\tilde{\mathcal{L}}$
on $\mathcal{C}_s$ descends to a line bundle $L_\mathrm{par}$ on $\mathrm{M}^s_X(\tau_p)$. 
Recall that the Lagrangian quotient 
$\psi \colon \mathcal{C}_\mathrm{s}^{\alpha_\sigma} \ra \mathcal{N}_{\tilde{\sigma}}^{\tau_p}$ is a 
$\mathcal{PG}_\mathrm{par}^{\gamma_\sigma}$-principal bundle. Then, the restriction of the
line bundle $\tilde{\mathcal{L}}$ to $\mathcal{C}_\mathrm{s}^{\alpha_\sigma}$ descends to a line bundle
$L_{\tilde{\sigma}}^{\tau_p}$ on $\mathcal{N}_{\tilde{\sigma}}^{\tau_p}$.
Consider the following 
\[
\xymatrix{
\mathcal{C}_\mathrm{s}^{\alpha_\sigma} \ar[r]^\iota \ar[d]_\psi & \mathcal{C}_\mathrm{s} \ar[d]^\varphi \\
\mathcal{N}_{\tilde{\sigma}}^{\tau_p} \ar[r]_j & **[r]\mathrm{M}^s_X(\tau_p)
}
\]
commutative diagram of principal bundles. Then, we have $j^*(L_\mathrm{par}) \cong L_{\tilde{\sigma}}^{\tau_p}$.

%%%%%%%%%%%%%%%%%%%%%%%%%%%%%%%%%%%%%%%%%%%%%%%%%%%%%%%%%%%%%%%%%%%%%%%

%%%%%%%%%%%%%%%%%%%%%%%%%%%%%%%%%%%%%%%%%%%%%%%%%%%%%%%%%%%%%%%%%%%%%%%

\begin{thebibliography}{012345}
 \bibitem{AB82} M. F. Atiyah, R. Bott, \emph{The Yang-Mills equations over Riemann surfaces},
		 Philos. Trans. Roy. Soc. London Ser. A, \textbf{308}, 1982, p. 523-615.
 \bibitem{Am14a} S.~Amrutiya, \emph{Connections on real parabolic bundles over a real curve},
 		 Bull. Korean Math. Soc., \textbf{51} (2014) 1101-1113.
 \bibitem{Am14b} S.~Amrutiya, \emph{Real parabolic vector bundles over a real curve},
 		 Proc. Indian Acad. Sci. Math. Sci., \textbf{124} (2014), 17-30.
 \bibitem{Bi91} O.~Biquard, \emph{Fibr\`es parabolique stables et connexions singuli\'eres plates}, 
		Bull. Soc. Math. France, \textbf{119} (1991), 231-257.
 \bibitem{BR93} I.~Biswas,~N.~Raghavendra, \emph{Determinants of parabolic bundles on Riemann surfaces}, 
 		 Proc. Indian Acad. Sci. Math. Sci. \textbf{103} (1) (1993) 41-71.
 \bibitem{BS20} I.~Biswas,~F.~Schaffhauser, \emph{Parabolic vector bundles on Klein surfaces},
		  Illinois J. Math. \textbf{64} (2020), no. 1, 105-118.
 \bibitem{BHH10} I.~Biswas,~J.~Huisman,~and~J.~Hurtubise, \emph{The
		 moduli space of stable vector bundles over a real algebraic curve},
		 Math. Ann., \textbf{347} (2010), 201-233.
 \bibitem{DW97}G.~D.~Daskalopoulos,~R.~A.~Wentworth, \emph{Geometric quantization for the moduli 
 		 space of vector bundles with parabolic structure},  Geometry, topology and physics (Campinas, 1996), 
 		 119–155, de Gruyter, Berlin, 1997
 \bibitem{Do83} S.~K.~Donaldson, \emph{A new proof of a theorem of Narasimhan and Seshadri},
 		 J. Differential Geom., \textbf{18} (1983), 269-277.
 \bibitem{Ko93} H.~Konno, \emph{Construction of the moduli space of stable parabolic Higgs bundles on a
		 Riemann surface}, J. Math. Soc. Japan, \textbf{45}(2) (1993) 253-276.
 \bibitem{LM85} R.~B.~Lockart,~R.~C.~McOwen, \emph{Elliptic differential operators on 
		 noncompact manifoids}, Ann. Scuola. Norm. Sup. Pisa Cl. Sci.(4), \textbf{12}, 1985, p. 409-447.
 \bibitem{MS80} V.~B.~Mehta,~C.~S.~Seshadri, \emph{Moduli of vector bundles
         on curves with parabolic structure}, Math. Ann.  \textbf{248} (1980), 205-239.
 \bibitem{NS65} M.~S.~Narasimhan,~C.~S.~Seshadri, \emph{Stable and unitary vector bundles on 
 		 a compact Riemann surface}, Ann. of Math. (2) \textbf{82} (1965), 540-567.
 \bibitem{NS95} E.~B.~Nasatyr,~B.~Steer, \emph{The Narasimhan–Seshadri theorem for parabolic bundles: 
 		 An orbifold approach}, Philos. Trans. Roy. Soc. A 353 (1995), no. 1702, 137-171.
 \bibitem{Po93} J.~A.~Poritz, \emph{Parabolic vector bundles and Hermitian–Yang–Mills 
 		 connections over a Riemann surface}, Internat. J. Math. \textbf{4} (1993), no. 3, 467–501.
 \bibitem{Q85} D.~Quillen, \emph{Determinants of Cauchy–Riemann operators over a Riemann surface}, 
 		 Funct. Anal. Appl. \textbf{19} (1985) 31-34. 
 \bibitem{Se70} C.~S.~Seshadri, \emph{Moduli of $\pi$--vector bundles over
         an algebraic curve},  1970  Questions on Algebraic Varieties
         (C.I.M.E., III Ciclo, Varenna, 1969)  pp. 139-260 Edizioni Cremonese, Rome
 \bibitem{Se77} C.~S.~Seshadri, \emph{Moduli of vector bundles on curves with
         parabolic structures}, Bull. Amer. Math. Soc. \textbf{83} (1977), 124--126
 \bibitem{Se82} C.~S.~Seshadri, \emph{Fibr\'es vectoriels sur les courbes
         alg\'ebriques}, Ast\'erisque \textbf{96}, Soc. Math. France, Paris, 1982
 \bibitem{Sch11} F.~Schaffhauser, \emph{Moduli spaces of vector bundles over a Klein surface},
 		 Geom. Dedicata \textbf{151} (2011) 187-206.
 \bibitem{Sch12} F.~Schaffhauser, \emph{Real points of coarse moduli schemes of vector bundles 
 		 on a real algebraic curve}, J. Symplectic Geom. \textbf{10} (2012), 503-534.
 \bibitem{Sch17} F.~Schaffhauser, \emph{On the Narasimhan–Seshadri correspondence for real
 		 and quaternionic vector bundles}, J. Differential Geom. \textbf{105} (2017), no. 1, 119-162.
 \bibitem{Si90} SIMPSON (C.T.). - \emph{Harmonic bundles on noncompact curves}, 
		J. Amer. Math. Soc., \textbf{3}, 1990, p. 713-770.

\bibitem{U82} K. K. Uhlenbeck, \emph{Connections with $L^p$ bounds on curvature}, 
		Comm. Math. Phys., \textbf{83}, 1982, p. 31-42.

\bibitem{UY86} K. K. Uhlenbeck, S. T. Yau, \emph{On the existence of Hermitian Yang-Mills
		connections in stable vector bundles}, Comm. Pure Appl. Math., \textbf{39-S}, 1986, p. 257-293.

\end{thebibliography}
\end{document}